\documentclass[11pt]{amsart}
\usepackage{bm}
\usepackage{amssymb}
\usepackage{graphicx}	
\usepackage{xcolor}
\usepackage{amsmath}
\usepackage{enumerate}
\usepackage{blindtext}
\usepackage{multirow}
\usepackage{mathrsfs}

\newtheorem*{notn}{Notation}

\newtheorem{thm}{Theorem}[section]

\newtheorem{cor}[thm]{Corollary}

\newtheorem{Def}[thm]{Definition}
\newtheorem{prop}[thm]{Proposition}
\newtheorem{rem}[thm]{Remark}
\newtheorem{ex}[thm]{Example}

\newcommand{\bdfn}{\begin{Def} \rm}
\newcommand{\edfn}{\end{Def}}

\newcommand{\tfae}{the following are equivalent}

\newcommand{\ra}{\rightarrow}
\newcommand{\Ra}{\Rightarrow}
\newcommand{\Lra}{\Leftrightarrow}

\newcommand{\es}{\emptyset}
\newcommand{\vp}{\varphi}
\newcommand{\ci}{\subseteq}

\newcommand{\al}{\alpha}
\newcommand{\be}{\beta}
\newcommand{\de}{\delta}
\newcommand{\e}{\varepsilon}

\newcommand{\De}{\Delta}

\newcommand{\si}{\sigma}
\newcommand{\ga}{\gamma}

\newcommand{\La}{\Lambda}
\newcommand{\mb}{\mathbb}
\newcommand{\mc}{\mathcal}
\newcommand{\tr}{\textrm}
\newcommand{\mf}{\mathfrak}
\newcommand{\cc}{\circ}
\newcommand{\mr}{\mathscr}
\newcommand{\sm}{\setminus}
\newcommand{\Si}{\Sigma}

\newcommand{\iy}{\infty}

\newcommand{\Om}{\Omega}

\newcommand{\TFAE}{The following are equivalent}
\newcommand{\trC}{\textrm{Cent}}
\newcommand{\trr}{\textrm{rad}}
\newcommand{\emC}{\emph{Cent}}
\newcommand{\emr}{\emph{rad}}

\newcommand{\ercp}{\bf{\emph{rcp}}}
\newcommand{\eSACP}{\bf{\emph{SACP}}}

\newcommand{\beqa}{\begin{eqnarray*}}
\newcommand{\eeqa}{\end{eqnarray*}}
\newcounter{cnt1}
\newcounter{cnt2}
\newcounter{cnt3}
\newcounter{cnt4}
\newcommand{\blr}{\begin{list}{$($\roman{cnt1}$)$} {\usecounter{cnt1}
\setlength{\topsep}{0pt} \setlength{\itemsep}{0pt}}}
\newcommand{\blR}{\begin{list}{\Roman{cnt4}.\ } {\usecounter{cnt4}
\setlength{\topsep}{0pt} \setlength{\itemsep}{0pt}}}
\newcommand{\bla}{\begin{list}{$(\alph{cnt2})$} {\usecounter{cnt2}
\setlength{\topsep}{0pt} \setlength{\itemsep}{0pt}}}
\newcommand{\bln}{\begin{list}{$($\arabic{cnt3}$)$} {\usecounter{cnt3}
\setlength{\topsep}{0pt} \setlength{\itemsep}{0pt}}}
\newcommand{\el}{\end{list}}

\begin{document}
\title[$\mr{F}$-simultaneous approximative $\tau$-compactness in Banach spaces]{A study on $\mr{F}$-simultaneous approximative $\tau$-compactness property in Banach spaces}
\author[Das]{Syamantak Das}
\author[Paul]{Tanmoy Paul}
\address{Dept. of Mathematics\\
Indian Institute of Technology, Hyderabad\\
India}
\email{ma20resch11006@iith.ac.in \& tanmoy@math.iith.ac.in}
\subjclass[2000]{Primary 41A28, 46B20 Secondary 41A65, 41A50 \hfill \textbf{\today} }
\begin{abstract}
Vesel\'y  (1997) studied Banach spaces that admit $f$-centers for finite subsets of the space. In this work, we introduce the concept of $\mr{F}$-simultaneous approximative $\tau$-compactness property ($\tau$-$\mr{F}$-SACP in short) for triplets $(X, V,\mf{F})$, where $X$ is a Banach space, $V$ is a $\tau$-closed subset of $X$, $\mf{F}$ is a subfamily of closed and bounded subsets of $X$, $\mr{F}$ is a collection of functions, and $\tau$ is the norm or weak topology on $X$.  We characterize reflexive spaces with the Kadec-Klee property using triplets with $\tau$-$\mr{F}$-SACP.  We investigate the relationship between $\tau$-$\mr{F}$-SACP and the continuity properties of the restricted $f$-center map. The study further examines  $\tau$-$\mr{F}$-SACP in the context of $CLUR$ spaces and explores various characterizations of $\tau$-$\mr{F}$-SACP, including connections to reflexivity, Fr\'echet smoothness, and the Kadec-Klee property. 
\end{abstract}
\maketitle
\section{Introduction}

\subsection{Prerequisites:~}
We list some standard notations used in this study: $X$ denotes a Banach space, and by a subspace $Y$ of $X$, we refer to a closed linear subspace. $B_X$ and $S_X$ represent the closed unit ball and unit sphere of $X$, respectively. $\mc{F}(X),~\mc{CC}(X)$ and $\mc{CB}(X)$ denote the collection of all nonempty finite, closed and convex, closed and bounded subsets of $X$, respectively. For $F\in\mc{F}(X)$, $card(F)$ denotes the cardinality of the set $F$. For $f\in S_{X^*}$, we define $J_X(f)=\{x\in S_X:f(x)=1\}$. $NA(X)$ denotes the set of all $f\in S_{X^*}$ such that $f$ attains its norm on $B_X$. The real line is assumed to correspond to the scalar field of the space. In this study, $\tau$ denotes the norm or weak topology on $X$. For $F\in\mc{CB}(X)$,  we define a subclass $\ell_\iy^+(F)$ by $\ell_\iy^+(F)=\{\vp:F\ra\mb{R}_{\geq0}:\underset{t\in F}{\sup}~\vp(t)<\iy\}$. On the set $\ell_\iy^+(F)$, we consider a coordinate-wise ordering defined as: for $\vp_1,\vp_2\in\ell_\iy^+(F)$, $\vp_1\leq\vp_2$ if and only if $\vp_1(t)\leq\vp_2(t)$ for all $t\in F$.  
\bdfn
For $F\in\mc{CB}(X)$ and a function $f:\ell_\iy^+(F)\ra\mb{R}_{\geq 0}$, we state that
\bla
\item $f$ is {\it monotone} if for $\vp_1,\vp_2\in\ell_\iy^+(F)$, $\vp_1\leq\vp_2$ implies $f(\vp_1)\leq f(\vp_2)$.
\item $f$ is {\it weakly strictly monotone} if $f$ is monotone and for $\vp_1,\vp_2\in\ell_\iy^+(F)$, $\vp_1(t)<\vp_2(t)$ for all $t\in F$ implies $f(\vp_1)<f(\vp_2)$.
\item $f$ is {\it coercive} if $f(\vp)\ra\iy$ as $\|\vp\|_\iy\ra\iy$.
\el
\edfn

For $\mf{F}\ci\mc{CB}(X)$, we define the following.
\begin{notn}\quad
\bln\item[\rm(1)] $\mr{F}_{cm}$ denotes the collection of all convex monotone functions $f:\ell_\iy^+(F)\ra\mb{R}_{\geq0}$ for all $F\in\mf{F}$.
\item[\rm(2)] $\mr{F}_{cwm}$ denotes the collection of all convex weakly strictly  monotone functions $f:\ell_\iy^+(F)\ra\mb{R}_{\geq0}$ for all $F\in\mf{F}$.
\item[\rm(3)] $\mr{F}_{cmc}$ denotes the class of all convex monotone and coercive functions $f:\ell_\iy^+(F)\ra\mb{R}_{\geq0}$ for all $F\in\mf{F}$. 
\item[\rm(4)] $\mr{F}_{cwmc}$ denotes the collection of all convex weakly strictly monotone coercive functions $f:\ell_\iy^+(F)\ra\mb{R}_{\geq0}$ for all $F\in\mf{F}$. 
\el
\end{notn}
For $x\in X$ and a $\tau$-closed subset $V$ of $X$, we define $d(x,V)=\underset{v\in V}{\inf}\|x-v\|$ and $P_V(x)=\{v\in V:\|x-v\|=d(x,V)\}$. We note that $P_V(x)$ may be empty. $V$ is said to be {\it proximinal} ({\it Chebyshev}) in $X$ if $P_V(x)$ is nonempty (singleton) for all $x\in X$.

Let us recall the definitions of $\tau$-strong proximinality and approximative $\tau$-compactness from \cite{PB}.
\bdfn
\bla
\item A $\tau$-closed subset $V$ of $X$ is said to be {\it $\tau$-strongly proximinal} in $X$ if for all $x\in X$, $P_V(x)\neq\es$ and  for a $\tau$-neighbourhood $W$ of 0, there exists $\de>0$ such that $P_V(x,\de)\ci P_V(x)+W$, where $P_V(x,\delta)=\{v\in
V:\|v-x\|<d(x,V)+\delta\}$.
\item A $\tau$-closed subset $V$ of $X$ is said to be {\it approximatively $\tau$-compact} in $X$, if for $x\in X$ and a sequence $(v_n)\ci V$ such that $\|x-v_n\|\ra d(x,V)$ implies $(v_n)$ has a $\tau$-convergent subsequence.
\el
\edfn
For $x\in X$, a $\tau$-closed subset $V$of $X$, $F\in\mc{CB}(X)$, and a function $f:\ell_\iy^+(F)\ra\mb{R}_{\geq0}$, we define the following.
\begin{notn}\quad
\bln\item[\rm(1)] $r_f(x,F) = f((\|x-a\|)_{a\in F})$.
\item[\rm(2)] $\emph{rad}_V^f(F) = \underset{v\in V}{\inf}r_f(v,F)$.
\item[\rm(3)] $\emph{Cent}_V^f(F) = \{v\in V:r_f(v,F)=\emph{rad}_V^f(F)\}$.
\item[\rm(4)] $\de-\emph{Cent}_V^f(F) = \{v\in V:r_f(v,F)\leq \emph{rad}_V^f(F)+\de\}$.

\el
\end{notn}
However for $V\in\mc{CC}(X)$, $F\in\mc{CB}(X)$ and a function $f:\ell_\iy^+(F)\ra\mb{R}_{\geq0}$, $\trC_V^f(F)$ may be empty, but for any $\de>0$, $\de-\trC_V^f(F)$ is always non-empty.

For $V\in\mc{CC}(X)$ and $F\in\mc{CB}(X)$, when $f:\ell_\iy^+(F)\ra\mb{R}_{\geq0}$ is of the form $f(a)=\underset{t\in F}{\sup}\rho(t)a(t)$, where $\rho=(\rho(t))_{t\in F}\in\ell_\iy^+(F)$, then the quantities $r_f(x,F)$, $\trr_V^f(F)$ and $\trC_V^f(F)$ are denoted by $r_\rho(x,F)$, $\trr_V^\rho(F)$ and $\trC_V^\rho(F)$ respectively. When $\rho(t)=1$ for all $t\in F$, then the above quantites are denoted by $r(x,F)$, $\trr_V(F)$ and $\trC_V(F)$ respectively, the so-called farthest distance from $x$ to $F$, the restricted Chebyshev radius of $F$ in $V$ and the restricted Chebyshev center of $F$ in $V$ respectively. We refer to \cite{DP,ST2,ST}, a few previous works of the authors in this line of investigation.

\bdfn\label{D1}
Let $V$ be a $\tau$-closed subset of $X$, $\mf{F}\ci\mc{CB}(X)$ and $\mr{F}$ be a collection of functions $f:\ell_\iy^+(F)\ra\mb{R}_{\geq0}$ for all $F\in\mf{F}$.
\bla
\item  The triplet $(X,V,\mf{F})$ is said to have {\it restricted $\mr{F}$-center property} ($\mr{F}$-rcp in short) if for all $F\in\mf{F}$ and functions $f:\ell_\iy^+(F)\ra\mb{R}_{\geq0}$, $\trC_V^f(F)\neq\es$. 

\item The triplet $(X,V,\mf{F})$ is said to have {\it $\tau$-$\mr{F}$-property-$(P_1)$} if, for all $F\in \mf{F}$ and all functions $f:\ell_\iy^+(F)\ra\mb{R}_{\geq0}$ in $\mr{F}$, $\trC_V^f(F)\neq\es$ and for all $\tau$-neighborhoods $W$ of $0$, there exists $\de>0$ such that $\de-\tr{Cent}_V^f(F)\ci\tr{Cent}_V^f(F)+W$.
\item The triplet $(X,V,\mf{F})$ is said to have {\it $\mr{F}$-simultaneous approximative $\tau$-compactness property} ($\tau$-$\mr{F}$-SACP in short) if, for all $F\in\mf{F}$, all functions $f:\ell_\iy^+(F)\ra\mb{R}_{\geq0}$ in $\mr{F}$ and any sequence $(v_n)\ci V$ such that $r_f(v_n,F)\ra \textrm{rad}_V^f(F)$ implies that $(v_n)$ has a $\tau$-convergent subsequence.
\item The triplet $(X,V,\mf{F})$ is said to have  {\it weighted simultaneous approximative $\tau$-compactness property} (weighted-$\tau$-SACP in short) if, for all $F\in\mf{F}$, all bounded weights $\rho=(\rho(t))_{t\in F}$, and any sequence $(v_n)\ci V$ such that $r_\rho(v_n,F)\ra \textrm{rad}_V^\rho(F)$ implies that $(v_n)$ has a $\tau$-convergent subsequence.
\el
\edfn
When $\tau$ is the norm topology on $X$, we simply say that $(X,V,\mf{F})$ has $\mr{F}$-property-$(P_1)$ (for $(b)$), $\mr{F}$-SACP ( for $(c)$) or weighted-SACP (for $(d)$).

It is clear from definition \ref{D1} that, if $\mc{F}_s$ is the collection of all singletons in $X$ and $f$ is the identity map on $\mb{R}_{\geq0}$, then a $\tau$-closed subset $V$ of $X$ is approximatively $\tau$-compact ($\tau$-strongly proximinal) in $X$ if and only if $(X,V,\mc{F}_s)$ has $\tau$-$\{f\}$-SACP ($\tau$-$\{f\}$-property-$(P_1)$).

Let us observe that for a $\tau$-closed subset $V$ of $X$, if the triplet $(X,V,\mc{F}(X))$ has $\tau$-$\mr{F}_{cm}$-SACP, then from \cite[Proposition~1.4 (i), (iv)]{LV}, it follows that $(X,V\mc{F}(X))$ has $\mr{F}_{cm}$-rcp.

It has been observed in \cite[Exmaple~2.1]{PB} that $c_0$ is not approximatively $\tau$-compact in $\ell_\iy$, however, it is well-known that $c_0$ is strongly proximinal in $\ell_\iy$. Thus, when $\mc{F}_s$ is the collection of all singletons in $\ell_\iy$ and $f$ is the identity map on $\mb{R}_{\geq0}$, the triplet $(\ell_\iy,c_0,\mc{F}_s)$ has $\{f\}$-property-$(P_1)$ but does not have $\tau$-$\{f\}$-SACP.

We refer to example \ref{E3} for an example of a subspace $Y$ of $X$ such that $Y$ is approximatively compact in $X$ but $(X,Y,\mc{F}(X))$ does not have $\tau$-$\mr{F}$-SACP.

Let $(\Om,\Si,\mu)$ be a finite measure space. For $1\leq p<\iy$, $L_p(\Si,X)$ denotes the space of all $\mu$-strongly measurable $X$-valued $p$-Bochner integrable functions $f:\Om\ra X$ with $\|f\|_p:=\left(\int_\Om\|f(t)\|^pd\mu(t)\right)^{\frac{1}{p}}$. 

Let us recall the notions of $L$-projections, $L$-summands, and $L$-embedded spaces from \cite{H}.
\bdfn
\bla
\item A linear projection $P:X\ra X$ is said to be an {\it $L$-projection} if 
$\|x\|=\|Px\|+\|x-Px\|$ for all $x\in X$.

\item A subspace $Y$ of $X$ is said to be an {\it $L$-summand} if $Y$ is the range of an $L$-projection.

\item $X$ is said to be {\it $L$-embedded} if $X$, under its canonical image in $X^{**}$, is the range of an $L$-projection on $X^{**}$. 
\el
\edfn
\cite[Ch. 4]{H} serves as a reference for examples and other properties of $L$-embedded spaces. It is shown in \cite[Theorem IV.1.5]{H} that, if $Y$ is the range of a norm-1 projection of an $L$-embedded space $X$, $Y$ is also $L$-embedded.

\bdfn\label{D2}
\bla 
\item \cite{LP}
$X$ is said to be {\it $\tau$-compactly locally uniformly rotund} ($\tau-CLUR$ in short) if, for every  $x\in S_X$ and every sequence $(x_n)\ci B_X$ satisfying $\|\frac{x_n+x}{2}\|\ra1$, $(x_n)$ has a $\tau$-convergent subsequence.
\item\cite{PB1}
$X$ is said to be {\it $\tau$-almost locally uniformly rotund} ($\tau-ALUR$ in short) if, for every $x,x_n\in S_X$ and $h_m\in S_{X^*}$ such that $\lim_m\lim_nh_m(\frac{x_n+x}{2})=1$, $x_n\xrightarrow{\tau}x$.
\item $X$ is said to have the {\it Kadec-Klee property} if $x,x_n\in S_X$ with $x_n\xrightarrow{w}x$ implies $x_n\ra x$ in the norm.
\el
It follows from \cite[Proposition 4.4 and Corollary 4.6]{PB1} that $X$ is $\tau-ALUR$ if and only if for every $f\in NA(X)$ and $(x_n)\ci B_X$ such that $f(x_n)\ra 1$ implies $(x_n)$ is $\tau$-convergent.
\edfn 
Let us recall property $(**)$ from \cite{NP} which later became known as the Namioka-Phelps property.
\bdfn
$X$ is said to have the {\it Namioka-Phelps} property if $w^*$  and norm convergence coincide on $S_{X^*}$.
\edfn
Let us recall the following definitions from \cite{Mg}.
\bdfn
Let $X$ be a Banach space.
\bla
\item The norm of $X$ is said to be {\it differentiable} at $x\in S_X$ if $\underset{t\ra0}{\lim}\frac{\|x+ty\|-1}{t}$ exists for $y\in X$. 
\item The norm of $X$ is said to be {\it Fr{\'e}chet differentiable} at $x\in S_X$ if the limit in $(a)$ exists and the convergence is uniform over $y\in S_X$.
\el
\edfn
A Banach space $X$ is said to be {\it smooth} ({\it Fr{\'e}chet smooth}) if the norm is differentiable (Fr{\'e}chet differentiable) at each $x\in S_X$.

\subsection{Observations:~}  In section~2, we study the relationship between $\tau$-$\mr{F}$-property-$(P_1)$ and $\tau$-$\mr{F}$-SACP for a triplet $(X,V,\mf{F})$. We study reflexivity and the Kadec-Klee property in relation to $\tau$-$\mr{F}$-SACP for a triplet $(X,V,\mf{F})$. In section~3, we derive that Fr\'echet smooth spaces are precisely those spaces in which for every $w^*$-closed convex subset $V$ of $X^*$, $(X^*,V,\mc{F}(X^*))$ has $\mr{F}$-SACP, where $\mr{F}$ is the collection of monotone smooth norms and $\trC_V^\ga(F)$ is singleton for all $F\in\mc{F}(X^*)$, and for all monotone norms $\ga:\mb{R}^N\ra\mb{R}_{\geq0}$ where $card(F)=N$. We further investigate several rotundity properties in relation to $\tau$-$\mr{F}$-SACP. 

In section~4, we study the continuity properties of certain set-valued mappings.

\section{Various characterizations of $\tau$-$\mr{F}$-SACP}\hfill

Let us recall that a norm $\ga$ on $\mb{R}^3$ is {\it symmetric} if $\ga=\ga\circ\pi$ for every permutation $\pi$ of the three variables.
\begin{ex}\label{E3}
Let $Y$ be a non-reflexive Banach space. Then from \cite[Theorem~2.4]{LV}, there exists a three-point set $F\ci X$ such that for any monotone symmetric norm $\ga:\mb{R}^3\ra\mb{R}_{\geq0}$, $Y$ admits an equivalent norm $|||~.~|||$ such that $\emC_Y^\ga(F)=\es$ in $(Y,|||~.~|||)$. Let $Y'$ denote the Banach space $Y$ equipped with the norm $|||~.~|||$, and let $X=Y'\oplus_{\ell_1}\mb{R}$. It is evident that $Y'$, being an $L$-summand, is approximatively compact in $X$. However, the triplet $(X,Y',\mc{F}(X))$ does not have $\tau$-$\mr{F}_{cm}$-$\eSACP$ as $\emC_{Y'}^\ga(F)=\es$. 
\end{ex}

The following result demonstrates the relationship between $\tau$-$\mr{F}_{cm}$-property-$(P_1)$ and $\tau$-$\mr{F}_{cm}$-SACP for a triplet $(X,V,\mc{A})$, where $V$ is a $\tau$-closed subset of $X$ and $\mc{A}\ci\mc{F}(X)$.

\begin{thm}\label{T2}
Let $V$ be a $\tau$-closed subset of $X$ and $\mc{A}\ci \mc{F}(X)$. Then, $(X,V,\mc{A})$ has $\tau$-$\mr{F}_{cm}$-$\eSACP$ if and only if $(X,V,\mc{A})$ has the $\tau$-$\mr{F}_{cm}$-property-$(P_1)$ and $\emC_V^f(F)$ is $\tau$-compact for all $F\in \mc{A}$ and $f:\ell_\iy^+(F)\ra\mb{R}_{\geq0}$ in $\mr{F}_{cm}$.
\end{thm}
\begin{proof}
The necessity follows from the proof of \cite[Thereom~2.2]{PB}. We now derive the converse.

Let $F\in\mc{A}$, $f:\ell_\iy^+(F)\ra\mb{R}_{\geq0}$ be a function in $\mr{F}_{cm}$ and $(v_n)\ci V$ be such that $r_f(v_n,F)\ra\trr_V^f(F)$. Let $W$ be a $\tau$-neighbourhood of $0$ and $W_1,W_2$ be two $\tau$-open neighbourhoods of $0$ such that $W_1+W_2\ci W$. As $(X,V,\mc{A})$ has the $\tau$-$\mr{F}_{cm}$-property-$(P_1)$, there exists $\de>0$ such that $\de-\trC_V^f(F)\ci\trC_V^f(F)+W_1$. Now, there exists $N_1\in\mb{N}$ such that $r_f(v_n,F)\leq\trr_V^f(F)+\de$ for all $n\geq N_1$, i.e., $v_n\in \de-\trC_V^f(F)$ for all $n\geq N_1$. Thus, there exists a sequence $(w_n)\ci\trC_V^f(F)$ such that $v_n-w_n\in W_1$ for all $n\geq N_1$. As $\trC_V^f(F)$ is $\tau$-compact, $(w_n)$ has a $\tau$-convergent subsequence $(w_{n_k})$ converging to some $w\in W$. Thus there exists $N_2\in\mb{N}$ such that $w_{n_k}\in \{w\}+W_2$ for all $k\geq N_2$ and so, $v_{n_k}=(v_{n_k}-w_{n_k})+w_{n_k}\in \{w\}+W$ for all $k\geq \max\{N_1,N_2\}$. This establishes our assertion.
\end{proof}

The following theorem characterizes reflexive spaces based on triplets satisfying weak-$\mr{F}_{cmc}$-SACP.

Using the fact that for a $F\in\mc{CB}(X)$ and a certain $f:\ell_\iy^+(F)\ra\mb{R}_{\geq0}$ in $\mr{F}_{cmc}$, due to the coercivity of $f$, any sequence $(v_n)\ci V$ such that $r_f(v_n,F)\ra\trr_V^f(F)$ implies that $(v_n)$ is bounded. Thus we have the following.
\begin{thm}\label{T4}
 If $X$ is reflexive, then for every $V\in\mc{CC}(X)$, the triplet $(X,V,\mc{CB}(X))$ has weak-$\mr{F}_{cmc}$-$\eSACP$.
\end{thm}
 Theorem~\ref{T4} leads to the following observation.
\begin{cor}\label{P2}
If $C$ is a weakly compact subset of $X$, then $(X,C,\mc{CB}(X))$ has weak-$\mr{F}_{cmc}$-$\eSACP$.
\end{cor}

\begin{cor}
Let $C$ be a weakly compact convex subset of $X$ and $1\leq p<\iy$. Then $(L_p(\Si,X),L_p(\Si,C),\mc{CB}(L_p(\Si,X))$ has weak-$\mr{F}_{cmc}$-$\eSACP$.
\end{cor}
\begin{proof}
Let us recall from \cite[Theorem~2]{JD} that if $C$ is a weakly compact convex subset of $X$, then for $1\leq p<\iy$, $L_p(\Si,C)$ is weakly compact in $L_p(\Si,X)$. Thus, from proposition \ref{P2}, our assertion follows.
\end{proof}
We now characterize reflexive spaces with the Kadec-Klee property.
\begin{thm}\label{T6}
For a Banach space $X$, \tfae.
\bla
\item $X$ is reflexive and has the Kadec-Klee property.
\item For every $V\in\mc{CC}(X)$, $(X,V,\mc{F}(X))$ has $\mr{F}_{cwmc}$-$\eSACP$.
\el
\end{thm}
\begin{proof}
It remains to prove $(a)\Ra (b)$. Let $F=\{x_1,\cdots,x_N\}\in\mc{F}(X)$, $f:\mb{R}^N_{\geq0}\ra\mb{R}_{\geq0}$ be a function in $\mr{F}_{cwmc}$ and  $(v_n)\ci V$ be such that $r_f(v_n, F)\ra\tr{rad}_V^f(F)$. Since $f$ is coercive, $(v_n)$ is bounded. Now, by the reflexivity of $X$, $(v_n)$ has a subsequence $(v_{n_k})$, converging weakly to $v_0\in V$. By passing on to a subsequence, if necessary, let $\mu_i=\underset{k}{\lim}\|v_{n_k}-x_i\|$ for all $i=1,\cdots,N$. Now, by the lower semicontinuity of the norm, we have $\|v_0-x_i\|\leq\mu_i$ for all $i=1,\cdots,N$. From \cite[Proposition~1.4(iv)]{LV}, by the lower semicontinuity of
the function $r_f(., F)$, we obtain
\[\trr_V^f(F)\leq r_f(v_0,F)\leq\liminf_k r_f(v_{n_k},F)=\trr_V^f(F).
\]
Thus $v_0\in\trC_V^f(F)$.

{\sc Claim:} There exists $i\in\{1,\cdots,N\}$ such that $\|v_0-x_i\|=\mu_i$.

If not, suppose that $\|v_0-x_i\|<\mu_i$ for all $i=1,\cdots,N$. Using the weakly
strictly monotonicity of $f$ and the continuity of $f$ on bounded subsets of $\mb{R}^N_{\geq0}$, we have
\beqa
f(\|v_0-x_1\|,\cdots,\|v_0-x_N\|)&<&f(\mu_1,\cdots,\mu_N)\\
&=&\lim_k f(\|v_{n_k}-x_1\|,\cdots,\|v_{n_k}-x_N\|)\\
&=&\trr_V^f(F).
\eeqa
This contradiction establishes our claim. Also, as $v_{n_k}-x_i\xrightarrow{w}v_0-x_i$, by the Kadec-Klee property of $X$, we have $v_{n_k}\ra v_0$. This establishes our assertion.
\end{proof}
\begin{rem}
The authors in \cite[Example~2.9]{ZZ} provided an example of a reflexive Banach space $X$ that does not have the Kadec-Klee property. Hence, from Theorems \ref{T4} and \ref{T6}, it follows that there exists a closed convex subset $V$ of $X$ such that $(X,V,\mc{F}(X))$ has weak-$\mr{F}_{cwmc}$-$\eSACP$ but $(X,V,\mc{F}(X))$ does not have $\mr{F}_{cwmc}$-$\eSACP$.
\end{rem}

Our next result shows that for Banach spaces $X$ with the Kadec-Klee property, for a  closed convex subset $V$ of $X$, weak-$\mr{F}_{cwm}$-SACP for a triplet $(X,V,\mc{F}(X))$ implies $\mr{F}_{cwm}$-SACP for $(X,V,\mc{F}(X))$.
\begin{prop}
For a Banach space $X$, \tfae.
\bla
\item $X$ has the Kadec-Klee property.
\item For $V\in\mc{CC}(X)$, if $(X,V,\mc{F}(X))$ has weak-$\mr{F}_{cwm}$-$\eSACP$, then  $(X,V,\mc{F}(X))$ has $\mr{F}_{cwm}$-$\eSACP$.
\el
\end{prop}
\begin{proof}
$(a)\Ra(b)$. Let $F\in\mc{F}(X)$, $f:\ell_\iy^+(F)\ra\mb{R}_{\geq0}$ be a function in $\mc{F}_{cwm}$ and  $(v_n)\ci V$ be such that $r_f(v_n, F)\ra\tr{rad}_V^f(F)$. Then, by our assumption, $(v_n)$ has a subsequence $(v_{n_k})$ converging weakly to $v_0\in V$. Then, as in the proof of $(a)\Ra(b)$ of Theorem~\ref{T6}, by passing on to a subsequence, if necessary, we can show that $v_{n_k}\ra v_0$, and so, $(a)\Ra(b)$. 

$(b)\Ra(a)$ Let $V\in\mc{CC}(X)$ be such that $(X,V,\mc{F}(X))$ has weak-$\mr{F}_{cwm}$-SACP. Then $V$ is weakly approximatively compact in $X$ and hence, by our assumption, $V$ is approximatively compact in $X$. Our assertion now follows from \cite[Page~396, Theorem~3]{HM}.
\end{proof}

Proceeding in the same way as in the proof of \cite[Theorem~3.4]{LV}, we obtain the following.
\begin{thm}
Let $X$ be a separable Banach space with the Namioka-Phelps property. Then, for any $w^*$-closed subspace $Y$ of $X^*$, $(X^*,Y,\mc{F}(X^*))$ has $\mr{F}_{cwmc}$-$\eSACP$. 
\end{thm}
\section{Various implications of smoothness and rotundity in relation to $\tau$-$\mr{F}$-SACP}\hfill

For a $\tau$-closed subset $V$ of $X$, and  a monotone norm $\ga:\mb{R}^N\ra\mb{R}_{\geq0}$ we define the following.
\begin{notn}
\bln
\item[\rm(1)] $\ell_\ga^N(X)$ denotes the space $X^N$ with norm $|||~.~|||$ defined by $|||(x_1,\cdots,x_N)|||=\ga(\|x_1\|,\cdots,\|x_N\|)$.
\item[\rm(2)]  $\De_\ga(V,N)=\{(v,\cdots, v)\in X^N:v\in V\}$. 
\el
\end{notn}
We further define the following.
\begin{notn}
\bln
\item[\rm(1)] $\mr{F}_{mn}$ denotes the class of all monotone norms $\ga:\mb{R}^N\ra\mb{R}_{\geq0}$ for all $N\in\mb{N}$.
\item[\rm(2)] $\mr{F}_{msn}$ denotes the class of all monotone smooth norms $\ga:\mb{R}^N\ra\mb{R}_{\geq0}$ for all $N\in\mb{N}$.
\item[\rm(3)] $\mr{F}_{mscn}$ denotes the class of all monotone strictly convex norms $\ga:\mb{R}^N\ra\mb{R}_{\geq0}$ for all $N\in\mb{N}$.
\el
\end{notn}
For $x\in X$, a $\tau$-closed subset $V$ of $X$, $F=\{x_1,\cdots,x_N\}\in\mc{F}(X)$, and a monotone norm $\ga:\mb{R}^N\ra\mb{R}_{\geq0}$, we have $r_\ga(x,F)=\ga(\|x-x_1\|,\cdots,\|x-x_N\|)=|||(x,\cdots,x)-(x_1,\cdots,x_N)|||$ and $\trr_V^\ga(F)=\underset{v\in V}{\inf}\ga(\|v-x_1\|,\cdots,\|v-x_N\|)=\underset{v\in V}{\inf}|||(v,\cdots,v)-(x_1,\cdots,x_N)|||=d((x_1,\cdots,x_N),\De_\ga(V,N))$. Now, it is easy to observe that $v\in\trC_V^\ga(F)$ if and only if  $(v,\cdots,v)\in P_{\De_\ga(V,N)}((x_1,\cdots,x_N))$. We identify the dual space $\ell_\ga^N(X)^*$ by $\ell_{\ga^\circ}^N(X^*)$, with the norm $|||~.~|||^\circ$ defined by $|||(f_1,\cdots,f_N)|||^\circ=\ga^\cc(\|f_1\|,\cdots,\|f_N\|)$,  where $\ga^\cc$ is the dual norm of $\ga$. The pairing between $\ell_\ga^N(X)$ and $\ell_{\ga^\cc}^N(X^*)$ is given as, for $f=(f_1,\cdots,f_N)\in\ell_{\ga^\cc}^N(X^*)$ and $x=(x_1,\cdots,x_N)\in\ell_\ga^N(X)$, $f(x)=\sum_{i=1}^Nf_i(x_i)$. 

From the above discussion, we immediately have the following.
\begin{thm}\label{T2}
Let $V$ be a $\tau$-closed subset of $X$ and $\mr{G}\ci\mr{F}_{mn}$. Then $(X,V,\mc{F}(X))$ has $\tau$-$\mr{G}$-$\eSACP$ if and only if $\De_\ga(V,N)$ is approximatively $\tau$-compact in $\ell_\ga^N(X)$ for all monotone norm  $\ga:\mb{R}^N\ra\mb{R}_{\geq0}$ in $\mr{G}$ for all $N\in\mb{N}$. Also for $F\in\mc{F}(X)$ with $card(F)=N$, and a monotone norm $\ga:\mb{R}^N\ra\mb{R}_{\geq0}$ in $\mr{G}$, $\emC_V^\ga(F)$ is singleton if and only if $P_{\De_\ga(V,N)}((x_1,\cdots,x_N))$ is singleton.
\end{thm}
Our next result shows that for every $N\in\mb{N}$ and a monotone smooth norm $\ga:\mb{R}^N\ra\mb{R}_{\geq0}$, $\ell_\ga^N(X)$ is Fr\'echet smooth, when $X$ is Fr\'echet smooth.
\begin{thm}\label{T3}
Suppose $X$ is Fr\'echet smooth. Then for every $N\in\mb{N}$ and monotone smooth norm $\ga:\mb{R}^N\ra\mb{R}_{\geq0}$, $\ell_\ga^N(X)$ is Fr\'echet smooth.
\end{thm}
\begin{proof}
We prove the case for $N=2$, since the proof for other values of $N$ follows similarly. Let $\ga:\mb{R}^2\ra\mb{R}_{\geq0}$ be a monotone smooth norm, and $\ga^\cc$ be the dual norm of $\ga$. Let $(x_1,x_2)\in S_{\ell_\ga^2(X)}$,  and $((f_n^1,f_n^2))\ci S_{\ell_{\ga^\cc}^2(X^*)}$ be such that $f_n^1(x_1)+f_n^2(x_2)\ra 1$. Let $(f_1,f_2)\in S_{\ell_{\ga^\cc}^2(X^*)}$ be such that $f_1(x_1)+f_2(x_2)=1$. Then $1=f_1(x_1)+f_2(x_2)\leq\|f_1\|\|x_1\|+\|f_2\|\|x_2\|\leq\ga^\cc(\|f_1\|,\|f_2\|)\ga(\|x_1\|,\|x_2\|)=1$. Thus $\|f_1\|\|x_1\|+\|f_2\|\|x_2\|=1$. Since $\ga$ is a smooth norm, $J_{\mb{R}^2}((\|x_1\|,\|x_2\|))=\{(\|f_1\|,\|f_2\|)\}$.

{\sc Claim:} $(\|f_n^1\|),(\|f_n^2\|)$ converges to $\|f_1\|$ and $\|f_2\|$ respectively.

Let $(\|f_{n_k}^1\|)$ be a subsequence of $(\|f_n^1\|)$. Since $(\|f_{n_k}^1\|)$ is bounded, it has a convergent subsequence, also denoted by $(\|f_{n_k}^1\|)$, converging to some $\al_1\in\mb{R}$. Similarly, $(\|f_{n_k}^2\|)$ has a convergent subsequence, also denoted by $(\|f_{n_k}^2\|)$, converging to some $\al_2\in\mb{R}$. Now $\al_1\|x_1\|+\al_2\|x_2\|=\underset{k}{\lim}(\|f_{n_k}^1\|\|x_1\|+\|f_{n_k}^2\|\|x_2\|)=1$. Since $\ga$ is a smooth norm on $\mb{R}^2$, we have $\al_1=\|f_1\|$ and $\al_2=\|f_2\|$.  Thus every subsequence of $(\|f_n^1\|)$ has a further subsequence converging to $\|f_1\|$ and hence $(\|f_n^1\|)$ is convergent to $\|f_1\|$. Proceeding in the same way, we can show that $(\|f_n^2\|)$ is convergent to $\|f_2\|$.

{\sc Case 1: } Suppose $x_1,x_2\neq0$. Then
\beqa
1=\underset{n}{\lim}(f_n^1(x_1)+f_n^2(x_2))&\leq&\underset{n}{\lim\inf}(\|f_n^1\|\|x_1\|+\|f_n^2\|\|x_2\|)\\
&\leq&\underset{n}{\lim\sup}(\|f_n^1\|\|x_1\|+\|f_n^2\|\|x_2\|)\\
&\leq&\underset{n}{\lim\sup}\ga^\cc(\|f_n^1\|,\|f_n^2\|)\ga(\|x_1\|,\|x_2\|)=1.
\eeqa
Hence $\underset{n}{\lim}(f_n^1(x_1)+f_n^2(x_2))=\underset{n}{\lim}(\|f_n^1\|\|x_1\|+\|f_n^2\|\|x_2\|)=1$.  Now it is easy to see that $\underset{n}{\lim}(\|f_n^1\|\|x_1\|-f_n^1(x_1))=0$ and  $\underset{n}{\lim}(\|f_n^2\|\|x_2\|-f_n^2(x_2))=0$. Hence $\frac{f_n^1}{\|f_n^1\|}(\frac{x_1}{\|x_1\|})\ra 1$ and $\frac{f_n^2}{\|f_n^2\|}(\frac{x_2}{\|x_2\|})\ra 1$. Now, due to the Fr\'echet differentiability of $X$, we have $(\frac{f_n^1}{\|f_n^1\|})$ and $(\frac{f_n^2}{\|f_n^2\|})$ are convergent.

 Hence $(f_n^1)$ and $(f_n^2)$ are convergent, and thus $((f_n^1,f_n^2))$ is convergent. 

 {\sc Case 2: } Suppose $x_1\neq0$ and $x_2=0$. Then we have $1=\|f_1\|\|x_1\|\leq\ga^\cc(\|f_1\|,0)\ga(\|x_1\|,\|x_2\|)=\ga^\cc(\|f_1\|,0)\leq\ga^\cc(\|f_1\|,\|f_2\|)=1$. Thus $\ga^\cc(\|f_1\|,0)=1$ and as $J_{\mb{R}^2}((\|x_1\|,\|x_2\|))=\{(\|f_1\|,\|f_2\|)\}$, we have $f_2=0$. Hence $(f_n^2)$ converges to 0. Now proceeding in a similar way, as in Case 1, we can show that $(f_n^1)$ is convergent. Thus $((f_n^1,f_n^2))$ is convergent. 

 This concludes our assertion.
\end{proof}
We now characterize  Banach spaces that are Fr\'echet smooth, based on triplets having $\mr{F}$-SACP.
\begin{thm}
For a Banach space $X$, the following are equivalent.
\bla
\item $X$ is Fr\'echet smooth. 
\item For any $w^*$-closed convex subset $V$ of $X^*$, $(X^*,V,\mc{F}(X^*))$ has  $\mr{F}_{msn}$-$\eSACP$ and for all $F\in\mc{F}(X^*)$ and monotone smooth norm $\ga:\mb{R}^N\ra\mb{R}_{\geq0}$, where $card(F)=N$, $\emC_V^\ga(F)$ is singleton.
\item Every $w^*$-closed convex subset of $X^*$ is approximatively compact and Chebyshev in $X^*$.
\el
\end{thm}
\begin{proof}
$(a)\Ra(c)$ Let $V$ be a $w^*$-closed convex subset of $X^*$ and $f\in X^*$. Without loss of generality, we may assume that $f=0$ and $d(0,V)=1$. Now, as $d(0,V)=1$, the open unit ball of $X^*$ is disjoint from the convex set $V$, and so, by \cite[Theorem~2.2.26]{Mg}, there exists $x\in S_X$ such that $f(x)\geq1$ for all $f\in V$. Let $(f_n)\ci V$ be such that $\|f_n\|\ra d(0,V)=1$. Now, $1\leq f_n(x)\leq \|f_n\|\ra 1$ and hence by the Fr\'echet differentiability of the norm in $X$ at $x$, $(\frac{f_n}{\|f_n\|})$ is convergent and hence $(f_n)$ is convergent. Since any minimizing sequence for $f$ in $V$ is convergent, $V$ is Chebyshev in $X^*$. This concludes our assertion.

$(c)\Ra(a)$ follows from \cite[Theorem~2.10]{PB}.

We only need to prove $(a)\Ra (b)$ as $(b)\Ra (c)$ is obvious.

$(a)\Ra(b)$ Let $V$ be a $w^*$-closed convex subset of $X^*$, $\ga:\mb{R}^N\ra\mb{R}_{\geq0}$ be a monotone smooth norm and  $\ga^\cc$ be the dual norm of $\ga$. Since $X$ is Fr\'echet smooth, by Theorem \ref{T3}, $\ell_\ga^N(X)$ is also Fr\'echet smooth. Now $\De_{\ga^\cc}(V,N)$ is a $w^*$-closed convex subset of $\ell_{\ga^\cc}^N(X^*)$ and hence from $(a)\Ra(c)$, $\De_{\ga^\cc}(V,N)$ is approximatively compact and Chebyshev in $\ell_{\ga^\cc}^N(X^*)$. Now, from Theorem \ref{T2}, our assertion follows.
\end{proof}
\begin{prop}\label{P5}
For an $N\in\mb{N}$ and monotone norm $\ga:\mb{R}^N\ra\mb{R}_{\geq0}$, if $(f_1,\cdots,f_N)\in NA(\ell_\ga^N(X))$, then $\frac{f_i}{\|f_i\|}\in NA(X)$ for all $i=1,\cdots,N$ such that $f_i\neq0$.
\end{prop}
\begin{proof}
We prove the result for $N=2$ as no new ideas are involved for higher values of $N$. Let $\ga:\mb{R}^2\ra\mb{R}_{\geq0}$ be a monotone norm and $\ga^\cc$ be the dual norm of $\ga$. Let  $(f_1,f_2)\in NA(\ell_\ga^2(X))$ and $(x_1,x_2)\in S_{\ell_\ga^2(X)}$ be such that $f_1(x_1)+f_2(x_2)=1$. Then
\beqa
1=f_1(x_1)+f_2(x_2)&\leq&\|f_1\|\|x_1\|+\|f_2\|\|x_2\|\\
&\leq&\ga^\cc(\|f_1\|,\|f_2\|)\ga(\|x_1\|,\|x_2\|)=1.
\eeqa
Thus $\|f_1\|\|x_1\|+\|f_2\|\|x_2\|=f_1(x_1)+f_2(x_2)$. Since $f_i(x_i)\leq\|f_i\|\|x_i\|$, we have $f_i(\frac{x_i}{\|x_i\|})=\|f_i\|$ for $i=1,2$. Thus if $f_i\neq0$, then $\frac{f_i}{\|f_i\|}\in NA(X)$ for $i=1,2$. 
\end{proof}
Our next result demostrates the fact that for every $N\in\mb{N}$, and a strictly convex monotone norm $\ga:\mb{R}^N\ra\mb{R}_{\geq0}$, $\ell_\ga^N(X)$ is a $\tau-ALUR$ space, when $X$ is a $\tau-ALUR$ space.
\begin{thm}\label{T9}
Let $X$ be a $\tau$-ALUR space. Then for every $N\in\mb{N}$, and strictly convex  monotone norm $\ga:\mb{R}^N\ra\mb{R}_{\geq0}$, $\ell_\ga^N(X)$ is a $\tau$-ALUR space.
\end{thm}
\begin{proof}
We consider the case when $N=2$ because the proof for other values of $N$ is similar. Let $\ga:\mb{R}^2\ra\mb{R}_{\geq0}$ be a strictly convex monotone norm and $\ga^\cc$ be the dual norm of $\ga$. Let $(f_1,f_2)\in NA(\ell_\ga^2(X))$ and $((x_n,y_n))\ci S_{\ell_\ga^2(X)}$ be such that $f_1(x_n)+f_2(y_n)\ra 1$.
Let $(z_1,z_2)\in S_{\ell_\ga^2(X)}$ be such that $f_1(z_1)+f_2(z_2)=1$.
 Then
\beqa
1=f_1(z_1)+f_2(z_2)&\leq&\|f_1\|\|z_1\|+\|f_2\|\|z_2\|\\
&\leq&\ga^\cc(\|f_1\|,\|f_2\|)\ga(\|z_1\|,\|z_2\|)=1.
\eeqa
Thus $\|f_1\|\|z_1\|+\|f_2\|\|z_2\|=1$. Thus $(\|f_1\|,\|f_2\|)$ attains its norm at $(\|z_1\|,\|z_2\|)$. 

{\sc Claim: } $(\|x_n\|), (\|y_n\|)$ converges to $\|z_1\|$ and $\|z_2\|$ respectively.  

Let $(\|x_{n_k}\|)$ be a subsequence of $(\|x_n\|)$. Since $(\|x_{n_k}\|)$ is bounded, $(\|x_{n_k}\|)$ has a convergent subsequence, also denoted by $(\|x_{n_k}\|)$, converging to some $\al\in\mb{R}$. In a similar way, we can show that $(\|y_{n_k}\|)$ has a convergent subsequence, also denoted by $(\|y_{n_k}\|)$, converging to some $\be\in\mb{R}$. Since the norm function is continuous, $\ga(\al,\be)=\underset{k}{\lim}\ga(\|x_{n_k}\|,\|y_{n_k}\|)=1$. Now $\|f_1\|\al+\|f_2\|\be=\underset{k}{\lim}(\|f_1\|\|x_{n_k}\|+\|f_2\|\|y_{n_k}\|)=1$. Now by the strict convexity of the norm $\ga$, we have $\al=\|z_1\|$ and $\be=\|z_2\|$. Thus $(\|x_n\|)$ is convergent to $\|z_1\|$ and $(\|y_n\|)$ is convergent to $\|z_2\|$.

{\sc Case 1: } Suppose $f_1,f_2\neq0$. Now 
\beqa 
1=\underset{n}{\lim}(f_1(x_n)+f_2(y_n))&\leq&\underset{n}{\lim\inf}(\|f_1\|\|x_n\|+\|f_2\|\|y_n\|)\\
&\leq&\underset{n}{\lim\sup}(\|f_1\|\|x_n\|+\|f_2\|\|y_n\|)\\
&\leq&\underset{n}{\lim\sup}\ga^\cc(\|f_1\|,\|f_2\|)\ga(\|x_n\|,\|y_n\|)\leq1.
\eeqa
Hence $\underset{n}{\lim}(\|f_1\|\|x_n\|+\|f_2\|\|y_n\|)=\underset{n}{\lim}(f_1(x_n)+f_2(y_n))=1$. Now it is easy to see that $\underset{n}{\lim}(\|f_1\|\|x_n\|-f_1(x_n))=0$ and $\underset{n}{\lim}(\|f_2\|\|y_n\|-f_2(y_n))=0$. Thus $\frac{f_1}{\|f_1\|}(\frac{x_n}{\|x_n\|})\ra 1$ and $\frac{f_2}{\|f_2\|}(\frac{y_n}{\|y_n\|})\ra1$.

From lemma \ref{P5}, $\frac{f_1}{\|f_1\|},\frac{f_2}{\|f_2\|}\in NA(X)$, and so, by our assumption, $(\frac{x_n}{\|x_n\|}),(\frac{y_n}{\|y_n\|})$ are $\tau$-convergent in $X$ and hence $((x_n,y_n))$ is $\tau$-convergent in $\ell_\ga^2(X)$.

{\sc Case 2:} Suppose $f_1\neq0$ and $f_2=0$. Now $1=\|f_1\|\|x_1\|=\ga^\cc(\|f_1\|,0)\ga(\|x_1\|,0)\leq \ga(\|x_1\|,\|x_2\|)=1$ and hence $\ga(\|x_1\|,0)=1$. Since $(\|f_1\|,0)$ attains its norm on $(\|x_1\|,0)$ and $\ga$ is strictly convex, we have $\|x_2\|=0$ and thus $x_2=0$. Hence $(y_n)$ converges to 0. Also, in a similar way, as in Case 1, $(x_n)$ is $\tau$-convergent. This concludes our assertion. 
\end{proof}
We now characterize $\tau-ALUR$ spaces, based on triplets having $\tau$-$\mr{F}$-SACP.
\begin{thm}
\TFAE.
\bla
\item $X$ is a $\tau$-ALUR space.
\item For every $V\in\mc{CC}(X)$, such that $(X,V,\mc{F}(X))$ has $\mr{F}_{scmn}$-$\ercp$, then $(X,V,\mc{F}(X))$ has $\tau$-$\mr{F}_{scmn}$-$\eSACP$ and for every $F\in\mc{F}(X)$ and strictly convex monotone norm $\ga:\mb{R}^N\ra\mb{R}_{\geq0}$, where $card(F)=N$, $\emC_V^\ga(F)$ is singleton.
\item Every proximinal closed convex subset of $X$ is approximatively compact and Chebyshev in $X$.
\el
\end{thm}
\begin{proof}
$(a)\Lra(c)$ follows from \cite[Theorem~2.11]{PB}.

We only need to prove $(a)\Ra(b)$ as $(b)\Ra(c)$ is obvious.

$(a)\Ra(b)$ Let $V\in\mc{CC}(X)$ be such that $(X,V,\mc{F}(X))$ has $\mr{F}_{scmn}$-rcp. Let $F\in\mc{F}(X)$, $card(F)=N$, and $\ga:\mb{R}^N\ra\mb{R}_{\geq0}$ be a strictly convex monotone norm. Then from Theorem \ref{T9}, $\ell_\ga^N(X)$ is a $\tau-ALUR$ space. Now by our assumption, $\De_\ga(V,N)$ is proximinal in $\ell_\ga^N(X)$ and hence from $(a)\Ra(c)$, $\De_\ga(V,N)$ is approximatively compact and Chebyshev in $\ell_\ga^N(X)$. This concludes our assertion.
\end{proof}
When $X$ is a $\tau-CLUR$ space, we have the following.
\begin{thm}
Let $X$ be a $\tau-CLUR$ space, $V\in\mc{CC}(X)$, $\mc{G}=\{F\in\mc{F}(X):\emC_V^\rho(F)\neq\es~\tr{for~all}~ \rho=(\rho_1,\cdots,\rho_{card(F)})\}$. Then, $(X,V,\mc{G})$ has weighted $\tau$-$\eSACP$.
\end{thm}
\begin{proof}
Let $F=\{x_1,\cdots,x_N\}\in \mc{G}$ and $\rho=(\rho_1,\cdots,\rho_N)$ be positive weights. Let $(v_n)\ci V$ be such that $r_\rho(v_n,F)\ra\trr_V^\rho(F)$. Then, it is easy to observe that $(v_n)$ is bounded. Let $v_0\in\trC_V^\rho(F)$ and $w_n=\frac{v_n+v_0}{2}$ for all $n\in\mb{N}$. By passing on to a subsequence, if necessary, let $r_\rho(w_n,F)=\rho_n\|w_n-x_n\|$, where $\rho_n\in(\rho_1,\cdots,\rho_N)$ and $x_n\in\{x_1,\cdots,x_N\}$ for all $n\in\mb{N}$. Then, there exists a subsequence $(x_{n_k})$ of $(x_n)$ such that $x_{n_k}=x_j$ for all $k\in\mb{N}$ and for some $j\in\{1,\cdots,N\}$. Thus, $r_\rho(w_{n_k},F)=\rho_j\|w_{n_k}-x_j\|$ for all $k\in\mb{N}$. Let $u_0=\rho_j(v_0-x_j)$ and $u_k=\rho_j(v_{n_k}-x_j)$ for all $k\in\mb{N}$. Then it is easy to observe that $\|u_0\|\leq\trr_V^\rho(F)$ and $\underset{k}{\lim}\|u_k\|\leq\trr_V^\rho(F)$. Now, for all $k\in\mb{N}$,
\[\|\frac{u_k+u_0}{2}\|=\rho_j\|w_{n_k}-x_j\|=r_\rho(w_{n_k},F)\geq\trr_V^\rho(F).\]
Since $X$ is $\tau-CLUR$, $(u_k)$ has a $\tau$-convergent subsequence, and so $(v_n)$ has a $\tau$-convergent subsequence. Hence, our conclusion follows.
\end{proof}

\begin{thm}\label{5}
Let $X$ be a Banach space, which is an $L$-summand in its bidual, and $Y$ be a subspace of $X$, which is also an $L$-summand in its bidual. Then, $(X,Y,\mc{F}(X))$ has weak-$\mr{F}_{wcmc}$-$\eSACP$.
\end{thm}
\begin{proof}
Let $F=\{x_1,\cdots,x_N\}\in \mc{F}(X)$, $f:\mb{R}^N_{\geq0}\ra\mb{R}_{\geq0}$ be a function in $\mr{F}_{wcmc}$  and $(y_n)\ci Y$ be such that $r_f(y_n,F)\ra\tr{rad}_Y^f(F)$. Let $P:X^{**}\ra X$ be the $L$-projecton. Then, from \cite[Theorem~1.2]{H}, we have $P(\overline{Y}^{w^*})=Y$. Since $f$ is coercive, $(y_n)$ is bounded. Let $\La$ be a $w^*$-accumulation point of $(y_n)$. Then, as in the proof of \cite[Theorem~3.5(a)]{ST}, we have $\trr_Y^f(F)=\trr_{\overline{Y}^{w^*}}^f(F)$. Since by \cite[Thorem~1.4(iv)]{LV}, $r_f(.,F):X^{**}\ra\mb{R}_{\geq0}$ is $w^*$-lower semicontinuous, we have 
\[\trr_{\overline{Y}^{w^*}}^f(F)\leq r_f(P(\La),F)\leq r_f(\La,F)\leq\trr_Y^f(F).\]
Thus, $r_f(\La,F)=r_f(P(\La),F)$. Now, if $\La\neq P(\La)$, then for all $i=1,\cdots,N$, we have
\[\|\La-x_i\|=\|P(\La)-x_i\|+\|\La-P(\La)\|>\|P(\La)-x_i\|.\]

Since $f$ is weakly strictly monotone, we have $r_f(P(\La),F)<r_f(\La,F)$, which is a contradiction. Thus $P(\La)=\La$, and so, $\La\in Y$. Thus, every $w^*$-accumulation point of $(y_n)$ is contained in $Y$, and so, $(y_n)$ is relatively weakly compact. This establishes our assertion.
\end{proof}
Let us recall from \cite[Chapter~5]{JD} that for a sub $\si$-algebra $\Si'\ci \Si$, the conditional expectation $E:L_1(\Si,X)\ra L_1(\Si',X)$ is defined by $E(f)=g$, where $\int_Bfd\mu=\int_Bgd\mu$ for all $B\in\Si'$ and it is a linear projection of norm-1. Thus, for a Banach space $X$ when $L_1(\Si,X)$ is $L$-embedded, then for a sub $\si$-algebra $\Si'\ci \Si$, $L_1(\Si',X)$ is also $L$-embedded. Thus we have the following.
\begin{cor}
Let $L_1(\Si,X)$ be an $L$-embedded space. Then for all sub $\si$-algebra $\Si'\ci \Si$, $(L_1(\Si,X),L_1(\Si',X),\mc{F}(L_1(\Si,X))$ has weak-$\mr{F}_{cwmc}$-$\eSACP$.
\end{cor}

\begin{thm}
Let $V$ be a closed subset of $X$, $F=\{x_1,\cdots,x_N\}\in \mc{F}(X)$ and $f:\mb{R}^N_{\geq0}\ra\mb{R}_{\geq0}$ be a convex monotone function. Then, for $(v_n)\ci V$ such that $r_f(v_n,F)\ra\emr_V^f(F)$ implies that $(v_n)$ is convergent if and only if, for all $\e>0$, there exists $\de>0$ such that $diam(\de-\emC_V^f(F))<\e$.
\end{thm}
\begin{proof}
 Suppose for all $(v_n)\ci V$ such that $r_f(v_n,F)\ra\trr_V^f(F)$ implies that $(v_n)$ is convergent. Then, it is easy to observe that $\trC_V^f(F)$ is singleton. Let $\trC_V^f(F)=\{v_0\}$. Suppose that there exists $\e>0$ such that for all $n\in\mb{N}$, $v_n\in\frac{1}{n}-\trC_V^f(F)$ but $\|v_n-v_0\|\geq\e$. Then $r_f(v_n,F)\ra\trr_V^f(F)$ but $(v_n)$ does not converge to $v_0$, which is a contraction. 

 Conversely, assume that, for all $\e>0$, there exists $\de>0$ such that $diam(\de-\trC_V^f(F))<\e$. Then, by the Cantor Intersection Theorem, $\trC_V^f(F)$ is singleton $\{v_0\}$ (say). Let $(v_n)\ci V$ be such that $r_f(v_n,F)\ra\trr_V^f(F)$. Then, there exists $N\in\mb{N}$, $v_n\in\de-\trC_V^f(F)$ for all $n\geq N$. Thus, $\|v_n-v_0\|\leq diam(\de-\trC_V^f(F))<\e$ for all $n\geq N$. Thus, $v_n\ra v_0$ and this establishes our assertion.
\end{proof}
\section{Continuity properties of some set-valued maps in relation to $\tau$-$\mr{F}$-SACP}
Throughout this section, by $B[x,r]$ we denote the closed ball centered at $x$ and radius $r$. In the following propositions, we present some basic observations.
\begin{prop}\label{P4}
Let $F=\{x_1,\cdots,x_N\}\in\mc{F}(X)$ and $f:\mb{R}^N_{\geq0}\ra\mb{R}_{\geq0}$ be a convex monotone function. Then, the map $r_f(.,F):X\ra\mb{R}_{\geq0}$ is Lipschitz continuous on bounded subsets of $X$.
\end{prop}

\begin{prop}\label{P1}
Let $Y$ be a subspace of $X$, $F=\{x_1,\cdots,x_N\}\in\mc{F}(X)$, $f:\mb{R}^N_{\geq0}\ra\mb{R}_{\geq0}$ be a monotone function, and $\e>0$. Then,
\bla
\item $\emr_Y^f(F)=\inf\{r_f(y,F):y\in\cup_{i=1}^NB[x_i,\emph{rad}_Y(F)+\e]\cap Y\}
=\inf\{r_f(y,F):y\in Y~\tr{and}~ \|y\|\leq\underset{1\leq i\leq N}{\max}\|x_i\|+\emph{rad}_Y(F)+\e\}$.
\item $\emC_Y^f(F)\ci \cup_{i=1}^NB[x_i,\emph{rad}_Y(F)+\e]\cap Y$ if $f$ is weakly strictly monotone.
\item $\emC_Y^f(F)\ci\cap_{i=1}^NB[x_i,\emph{rad}_Y(F)+diam(F)+\e]\cap Y$ if $f$ is weakly strictly monotone.
\el
\end{prop}
\begin{proof}
$(a)$. Let $y\in Y\sm(\cup_{i=1}^NB[x_i,\tr{rad}_Y(F)+\e]\cap Y)$. Then, $\|y-x_i\|>\tr{rad}_Y(F)+\e$ for all $i=1,\cdots,N$. We select $y_0\in Y$ such that $\underset{1\leq i\leq N}{\max}\|y_0-x_i\|\leq\trr_Y(F)+\e$. Then, we have $\|y_0-x_i\|\leq\trr_Y(F)+\e<\|y-x_i\|$ for all $i=1,\cdots,N$. As $f$ is monotone, $r_f(y_0,F)\leq r_f(y,F)$, and hence this establishes our first equality in $(a)$.

Let $y\in \cup_{i=1}^NB[x_i,\trr_Y(F)+\e]\cap Y$. Then, there exists $j\in\{1,\cdots,N\}$ such that $\|y-x_j\|\leq\trr_Y(F)+\e$. Thus, we have 
\[\|y\|\leq\|y-x_j\|+\|x_j\|\leq\trr_Y(F)+\underset{1\leq i\leq N}{\max}\|x_i\|+\e.\]
This establishes our second equality in $(a)$.

$(b)$. Let $y\in Y\sm(\cup_{i=1}^NB[x_i,\trr_Y(F)+\e]\cap Y)$. Then, as shown in the proof of $(a)$, we obtain a $y_0\in Y$ such that $\|y_0-x_i\|\leq\trr_Y(F)+\e<\|y-x_i\|$ for all $i=1,\cdots,N$. As $f$ is weakly strictly monotone, we have $r_f(y_0,F)< r_f(y,F)$. Thus, our assertion is established.

$(c)$. Let $y\in\trC_Y^f(F)$. Then, by $(b)$, there exists $j\in\{1,\cdots,N\}$ such that $\|y-x_j\|\leq\trr_Y(F)+\e$. Then, for all $i=1,\cdots,N$, we have
\[\|y-x_i\|\leq\|y-x_j\|+\|x_j-x_i\|\leq\trr_Y(F)+diam(F)+\e.\]
This establishes our claim.
\end{proof}
Let $\mc{F}_N(X)$ denote the set of all finite subsets of $X$ with cardinality $N$. 
We define a metric $d$ on the set $\mc{F}_N(X)$ as follows:  for $F_1,F_2\in\mc{F}_N(X)$,
\[
d(F_1,F_2)=\underset{1\leq i\leq N}{\max}\|x_i-x_i'\|
\]
where $F_1=\{x_1,\cdots,x_N\},F_2=\{x_1',\cdots,x_N'\}$.
\begin{prop}
Let $Y$ be a subspace of $X$ and $f:\mb{R}^N_{\geq0}\ra\mb{R}_{\geq0}$ be a convex, monotone function. Then, the map $\emr_Y^f:(\mc{F}_N(X),d)\ra\mb{R}_{\geq0}$ is Lipschitz continuous on bounded subsets of $\mc{F}_N(X)$.
\end{prop}
\begin{proof}
Let $B$ be a closed and bounded subset of $\mc{F}_N(X)$. Now there exists  $M>0$  such that $d(F,F')\leq M$ for all $F,F'\in B$. Fix a set $G=\{z_1,\cdots z_N\}\in B$. Let $R>0$ be such that $\underset{1\leq i\leq N}{\max}\|z_i\|\leq R$.

{\sc Claim:} $\trr_Y^f$ is Lipschitz continuous on $B$.

Let $F=\{x_1,\cdots,x_N\}, F'=\{x_1',\cdots,x_N'\}\in B$,
Now, $\underset{1\leq i\leq N}{\max}\|x_i\|+\trr_Y(F)\leq 2\underset{1\leq i\leq N}{\max}\|x_i\|\leq 2\underset{1\leq i\leq N}{\max}\|x_i-z_i\|+2\underset{1\leq i\leq N}{\max}\|z_i\|\leq 2M+2R$, and from proposition \ref{P1}$(a)$, we have $\trr_Y^f(F)=\inf\{r_f(y,F):\|y\|\leq 2M+2R+\e\}$ for an $\e>0$. Thus, for any $H=\{s_1,\cdots,s_N\}\in B$ and any $y\in Y$ such that $\|y\|\leq 2M+2R+\e$, we have $\|y-s_i\|\leq 3M+3R+\e$ for all $i=1,\cdots,N$. For $\e>0$, let $L$ be the Lipschitz constant of $f$ on $[0,3M+3R+\e]^N$. Then, for  $y\in Y$ with $\|y\|\leq 2M+2R+\e$, we have
\begin{eqnarray}\label{E1}
\begin{split}
|f(\|y-x_1\|,\cdots,\|y-x_N\|)-f(\|y-x_1'\|,\cdots,\|y-x_N'\|)|\\
\leq L\underset{1\leq i\leq N}{\max}|\|y-x_i\|-\|y-x_i'\||
\leq L\underset{1\leq i\leq N}{\max}\|x_i-x_i'\|.
\end{split}
\end{eqnarray}
Thus,
\beqa
\trr_Y^f(F)&\leq& f(\|y-x_1\|,\cdots,\|y-x_N\|)\\
&\leq&f(\|y-x_1'\|,\cdots,\|y-x_N'\|)+L\underset{1\leq i\leq N}{\max}\|x_i-x_i'\|.
\eeqa
Hence $\trr_Y^f(F)\leq\trr_Y^f(F')+L\underset{1\leq i\leq N}{\max}\|x_i-x_i'\|$. Similarly, by interchanging $F$ and $F'$ in (\ref{E1}), we have $|\trr_Y^f(F)-\trr_Y^f(F')|\leq L\underset{1\leq i\leq N}{\max}\|x_i-x_i'\|=Ld(F,F')$. Hence, our assertion follows.
\end{proof}
We now recall the notion of upper semicontinuity for set-valued maps.
\bdfn
Let $T$ be a topological space and $F:T\ra \mc{CB}(X)$ be a set-valued map. Then, $F$ is said to be {\it upper semicontinuous} if, for any $t_0\in T$ and any open set $W$ containing $F(t_0)$, there exists an open set $U$ containing $t_0$ such that $F(t)\ci W$ for all $t\in U$.
\edfn
\begin{thm}
Let $Y$ be a subspace of $X$ and $(X,Y,\mc{F}(X))$ has $\tau$-$\mr{F}_{cwm}$-$\eSACP$. Then, for $N\in\mb{N}$ and $f:\mb{R}^N_{\geq0}\ra\mb{R}_{\geq0}$ in $\mr{F}_{cwm}$, the map $\emC_Y^f:(\mc{F}_N(X),d)\ra\mc{CB}(Y)$ is metric-$\tau$ upper semicontinuous. 
\end{thm}
\begin{proof}
Suppose that $\trC_Y^f$ is not metric-$\tau$ upper semicontinuous at $F=\{x_1,\cdots,x_N\}\in \mc{F}_N(X)$. Then, there exists a $\tau$-open set $W$ containing $\trC_Y^f(F)$ and a sequence $(F_n)\ci \mc{F}_N(X)$ such that $F_n\ra F$ but $\trC_Y^f(F_n)\sm W\neq\es$ for all $n\in\mb{N}$. 

Let $F_n=\{x_1^n,\cdots,x_N^n\}$ for all $n\in\mb{N}$. Let $r=\trr_Y^f(F), r_n=\trr_Y^f(F_n),\vp(x)=r_f(x,F)$ and $\vp_n(x)=r_f(x,F_n)$ for all $x\in X$ and $n\in\mb{N}$. Let us choose $M>0$ such that for all $n\in\mb{N}$, $\underset{1\leq i\leq N}{\max}\|x_i^n\|\leq M$ and $\underset{1\leq i\leq N}{\max}\|x_i\|\leq M$. Let $y_n\in\trC_Y^f(F_n)\sm W$ for all $n\in\mb{N}$ and $\e>0$. Then, from proposition \ref{P1}$(c)$, we have $\|y_n-x_i^n\|\leq \trr_Y(F_n)+diam(F_n)+\e\leq 3M+\e$, and hence, $\|y_n-x_i\|\leq\|y_n-x_i^n\|+\|x_i^n\|+\|x_i\|\leq 5M+\e$ for all $i=1,\cdots,N$. Let $L$ be the Lipschitz constant for $f$ on $[0,5M+\e]^N$ and that for $\trr_Y^f$ on $\{\{y_1,\cdots,y_N\}\in \mc{F}_N(X):\underset{1\leq i\leq N}{\max}\|y_i\|\leq M\}$. Then, we have
\[
r\leq \vp(y_n)=r+(r_n-r)+(\vp(y_n)-\vp_n(y_n))
\leq r+2L\sum_{i=1}^N\|x_i^n-x_i\|.\]

Thus, $\vp(y_n)\ra r$. By our assumption, $(y_n)$ has a $\tau$-convergent subsequence $(y_{n_k})$ converging to $y\in Y$ in $\tau$-topology. Clearly, $y\in\trC_Y^f(F)\ci W$ but $W$ is a $\tau$-open set, and $y_n\notin W$ for all $n\in\mb{N}$. This contradiction proves our assertion.
\end{proof}
\begin{rem}
It has been observed in \cite{DK} that there exists a one dimensional subspace $Y$ of a finite dimensional Banach space $X$ such that the metric projection $x\mapsto P_Y(x)$ is not continuous. Since $Y$ is finite dimensional, $(X,Y,\mc{F}(X))$ has $\mr{F}_{cmc}$-$\eSACP$. However, the map $x\mapsto \emC_Y(\{x\})=P_Y(x)$ is not continuous. 
\end{rem}


\begin{thebibliography}{99}
\bibitem{PB1} Pradipta Bandyopadhyay, Da Huang, Bor-Luh Lin,
{\em Rotund points, nested sequence of balls and smoothness in Banach spaces}, Comment. Math. (Prace Mat.) {\bf 44}(2) (2004), 163--186.
\bibitem{PB}Pradipta Bandyopadhyay,  Yongjin Li, Bor-Luh Lin, Darapaneni Narayana, {\em Proximinality in Banach spaces}, J. Math. Anal. Appl. {\bf 341}(1) (2008), 309--317.
\bibitem{DP} Soumitra Daptari, Tanmoy Paul, {\em On property- $(R_1)$  and relative Chebyshev centers in Banach spaces}, Numer. Funct. Anal. Optim. {\bf43}(4) (2022), 486--495.
\bibitem{ST2} Syamantak Das, Tanmoy Paul, {\em On property- $(R_1)$  and relative Chebyshev centers in Banach spaces---{II}}, Quaest. Math. {\bf 47}(2) (2024), 461--476.
\bibitem{ST}Syamantak Das, Tanmoy Paul, {\em A study on various generalizations of generalized centers (GC) in Banach spaces}. Optimization (2024), 1--19.
\bibitem{DK} Frank Deutsch, Petar Kenderov, {\em Continuous selections and approximate selection for set-valued mappings and applications to metric projections}, SIAM J. Math. Anal. {\bf 14}(1) (1983), 185--194.
\bibitem{D}J. Diestel, J. J. Uhl, {\em Vector measures}, Mathematical Surveys {\bf 15}, American Mathematical Society, Providence, R.I., (1977).
\bibitem{JD}J. Diestel, {\em Remarks on weak compactness in $L_1(\mu,X)$}, Glasgow Math. J. {\bf 18}(1) (1997), 87--91.
\bibitem{DS} S. Dutta and P. Shunmugaraj, {\em Strong Proximinality Of Closed Convex Sets}, J. Approx. Theory {\bf 163}(4) (2011), 547--553. 
\bibitem{H} P. Harmand, D. Werner, W. Werner, {\em M-ideals in Banach spaces and Banach algebras}, Lecture Notes in Mathematics {\bf 1547}, Springer-Verlag, Heidelberg, (1993).
\bibitem{LSZ} Zhenghua Luo, Longfa Sun, Wen Zhang, {\em A Remark on the Stability of Approximative Compactness}, J. Funct. Spaces (2016), Art. ID 2734947.
\bibitem{JM} J. Mach, {\em Continuity properties of Chebyshev centers}, J. Approx. Theory {\bf 29}(3) (1980), 223--230.
\bibitem{Mg} Robert E. Megginson, {\em An introduction to Banach space theory}, Graduate Text in Mathematics, {\bf 183}, Springer-Verlag, New York, 1998.
\bibitem{HM} H. N. Mhaskar, D. V. Pai, {\em Fundamentals of approximation theory}, Narosa Publishing House, New Delhi, (2000).
\bibitem{NP} I. Namioka and R. R. Phelps, {\em Banach spaces which are Asplund spaces}, Duke Math. J. {\bf 42}(4), 735--750 (1975).
\bibitem{IS} I. Singer, {\em Some remarks on approximative compactness}, Rev. Roumaine Math. Pures Appl. 9 (1964), 167--177.
\bibitem{LV} Libor Vesel\'{y}, {\em Generalized centers of finite sets in Banach spaces} Acta Math. Univ.  {\bf66}(1) (1997), 83--115.
\bibitem{LP} L. P. Vlasov, {\em Approximative properties of sets in normed linear spaces}, Uspehi Mat. Nauk {\bf 28}(6) (1973) 3--66.
\bibitem{ZZ} Zihou Zhang, Chunyan Liu, Yu Zhou, {\em Some examples concerning proximinality in Banach spaces}, J. Approx. Theory {\bf 200} (2015), 136--143.
\end{thebibliography}
\end{document}